\documentclass[11pt]{article}

\usepackage{amsfonts}
\usepackage{amsmath}
\usepackage{amssymb}
\usepackage[mathscr]{eucal}
\usepackage[usenames]{color}
\usepackage[dvipsnames]{xcolor}
\usepackage{hyperref}
\usepackage{soul}
\newcommand{\os}{ \mathrm{os}}
\newcommand{\cyc}{ \mathrm{cyc}}
\newcommand{\Z}{\mathbb{Z}}
\newcommand{\lcm}{\mathop{\mathrm{lcm}}}

\def\eod{\vrule height 6pt width 5pt depth 0pt}
\newenvironment{proof}{\noindent {\bf Proof:} \hspace{.2em}}
                      {\hspace*{\fill}{\eod}\medskip}

\newcommand{\comment}[1]{}

\newtheorem{theorem}{Theorem}[section]
\newtheorem{proposition}[theorem]{Proposition}
\newtheorem{corollary}[theorem]{Corollary}

\newtheorem{question}[theorem]{Question}
\newtheorem{remark}[theorem]{Remark}

\newtheorem{lemma}[theorem]{Lemma}

\begin{document}
\title{On the order sequence of a group}	
	
	\author{
		Peter J. Cameron \\
		School of Mathematics and Statistics, \\
		 University of St Andrews, \\
		North Haugh, St Andrews, 
		Fife, KY16 9SS, UK \\
		email: pjc20@st-andrews.ac.uk \\
		\and 
		Hiranya Kishore Dey \\ 
		Department of Mathematics,\\
		Indian Institute of Science, Bangalore\\
		Bangalore 560 012, India\\
		email: hiranya.dey@gmail.com
	}
	
	\maketitle 
	
\begin{abstract}
This paper provides a bridge between two active areas of research, the
spectrum (set of element orders) and the power graph of a finite group.

The \emph{order sequence} of a finite group $G$ is the list of orders of
elements of the group, arranged in non-decreasing order. Order sequences
of groups of order $n$ are ordered by elementwise \emph{domination}, forming a
partially ordered set. We prove a number of results about this poset, among
them the following.
\begin{itemize}
\item M.~Amiri recently proved that the poset has a unique maximal element, corresponding to the cyclic group.
We show that the product of orders in a cyclic group of order $n$
is at least $q^{\phi(n)}$ times as large as the product in any non-cyclic group,
where $q$ is the smallest prime divisor of $n$ and $\phi$ is Euler's function,
with a similar result for the sum.
\item The poset of order sequences of abelian groups of order $p^n$ is naturally
isomorphic to the (well-studied) poset of partitions of $n$ with its natural
partial order.
\item If there exists a non-nilpotent group of order $n$, then there exists
such a group whose order sequence is dominated by the order sequence of any
nilpotent group of order $n$.
\item There is a product operation on finite ordered sequences, defined
by forming all products and sorting them into non-decreasing order. The
product of order sequences of groups $G$ and $H$ is the order sequence of a
group if and only if $|G|$ and $|H|$ are coprime.
\end{itemize}
The paper concludes with a number of open problems.
\end{abstract} 
	
	
	
\section{Introduction}
\label{sec:intro}
                     	
Let $G$ be a finite group. H. Amiri, S. M. Jafarian Amiri and I. M. Isaacs in \cite{amiri-communication} defined the following function:
$$\psi(G)= \sum_{g \in G} o(g),$$ 
where $o(g)$ denotes the order of the element $g$.  They were able to prove the following:
                      	
\begin{theorem}
\label{thm:amiri-cia} 
For any finite group $G$ of order $n$, $\psi(G) \leq \psi(\mathbb{Z}_n)$ and equality holds if and only if $G \cong \mathbb{Z}_n,$ where $\mathbb{Z}_n$ is the cyclic group of order $n$.
\end{theorem}
                      	
That is, $\mathbb{Z}_n$ is the unique group of order $n$ with the largest value of $\psi(G)$ for groups of that order. Later S. M. Jafarian Amiri and M. Amiri in \cite{amiri-pureandApplied} and, independently, R. Shen, G. Chen, and C. Wu in \cite{shen-et-al} investigated the groups with the second largest value of the sum of element orders. 
 
This function $\psi$ has been considered in various works (see \cite{amiri-communication-secondmax, amiri-algapplctn, asad-joa, chew-chin-lim, Her-pureandApplied,  Her-joa, Her-cia, Her-jpaa,  tarnauceanu-israel}). While the goal of some of the papers was to find out the largest, second largest, or least possible values of $\psi(G)$, others aimed  to prove new criteria for structural properties (like solvability, nilpotency, etc.) of finite groups.
                      	
Later, S. M. Jafarian Amiri and M. Amiri \cite{amiri-amiri-cia} considered the following  generalization of the above function defined by
$\psi_k(G)= \sum_{g \in G} o(g)^k$
for positive integers $k \geq 1$ and they proved that for any positive integer $k$, $\psi_k(G) < \psi_k(\mathbb{Z}_n)$  for all non-cyclic groups $G$ of order $n$. 
Recently, the product of the element orders of a group, which is denoted by $\rho(G)$, has also been considered and several results regarding $\rho(G)$ were proved in \cite{garonzi-patassini}, including the following:

\begin{theorem}
\label{thm:prod-thm-max} 
$\rho(G) \leq \rho(\mathbb{Z}_n)$ for every finite group $G$ of order $n$, and $\rho(G) = \rho(\mathbb{Z}_n)$ if and only if $G \cong \mathbb{Z}_n.$ 
\end{theorem}
                      
Inspired by these works and with a goal to give a unified approach to study those functions, we study the $\emph{order sequence}$ of a group $G$, which is defined as the sequence 
$$\os(G)=(o(g_1), o(g_2), \dots, o(g_n)),$$ 
where $o(g_i) \leq o(g_{i+1})$ for $1 \leq i \leq n-1.$ 

For example, one can check that
$\os(\mathbb{Z}_6)=(1,2,3,3,6,6)$ and $\os(S_3)=(1,2,2,2,3,3)$. 

The sequence $\os(\mathbb{Z}_n)$ for the cyclic group of order~$n$ can be determined explicitly: for each divisor $d$ of $n$, the entry $d$ occurs $\phi(d)$ times, where $\phi$ is Euler's function. 
                      	
Let $E(p^r)$ denote the elementary abelian group of order $p^r$. Then the order sequence of this group is $(1,p,p,\dots,p).$ Note that, if $p$ is odd and $r \geq 3$ then there are groups with the same order sequence as $E(p^r)$ but not isomorphic to it (non-abelian groups of exponent~$p$). The smallest examples of pairs of groups with the same order sequence have order~$16$; one such pair is
$\Z_4\times\Z_4$ and $\Z_2\times Q_8$, where $Q_8$ is the quaternion group.
                      	
For two groups $G$ and $H$ of order $n$, we say that $\os(G)$ \emph{dominates} $\os(H)$ if $o(g_i) \geq o(h_i)$ for $1 \leq i \leq n.$ This relation is a partial order, but not a total order; there are groups which are incomparable in this order. We shall implicitly think of it as an order on the finite groups of order $n$; it induces a \emph{partial preorder} on isomorphism classes of groups, that is, a reflexive and transitive relation, since there are groups which have the same order sequence.

However, there is a unique maximal element, namely $\Z_n$ (a recent result of 
M.~Amiri~\cite{amiri-strongdom}). Theorems \ref{thm:amiri-cia}
and  \ref{thm:prod-thm-max} follow immediately from this theorem. 
In Section \ref{sec:max-order-seq} we use the maximality
of $\Z_n$ under domination to establish bounds on the gaps between the values
of $\psi$ and of $\rho$ on cyclic and non-cyclic groups, and characterize
groups meeting these bounds (Theorem~\ref{thm:upper-bdd-prod-eltorder-G}).

In Section \ref{sec:products-extensions}, we define products of order sequences of two groups $G$ and $H$, and prove that if $|G|$ and $|H|$ are not coprime, then there exists no group whose order sequence is the same as the product of the order sequences of $G$ and $H$. We also show that the product of the order sequences of two groups $G$ and $H$ of coprime orders dominates the order sequence of any extension of $G$ by $H$ if $G$ is abelian. 

In Section \ref{sec:abelian-groups}, we explicitly determine the order sequence of an abelian $p$-group; using this, in Theorem \ref{thm:relation-with-poset-partitionlattice}, we show that the
poset of order sequences of abelian groups of order $p^n$ is naturally
isomorphic to the (well-studied) poset of partitions of $n$ with its natural
partial order. 

In Section \ref{sec:nilpotent-nonnilpotentgroups}, we find the groups with minimal order sequences among the family of nilpotent groups. In Theorem \ref{thm:non-nil-min-os-existence}, we show that if there is a non-nilpotent group of order $n$, 
 a group with minimal order sequence must be non-nilpotent.
 
Equality of the order sequence defines an equivalence relation on groups of
given order, and the equivalence classes are partially ordered by domination. In Section \ref{sec:abelian-groups}, we give a complete description of this order for abelian groups, in terms of the lattice of partitions of an integer. In
Section~\ref{sec:moregeneral}, we observe that for all groups the poset can be
rather complicated.

Section~\ref{sec:graphs} describes the connection between the order sequence
and some well-studied graphs related to groups. We show that the power graph
determines the order sequence, which in turn determines the Gruenberg--Kegel
graph.

The final Section~\ref{sec:more-questions} lists some open problems.
 
Throughout the paper, most of our notation is standard; for any undefined term, we refer the reader to the books \cite{Isac-ams, Scott}.

\section{Products and extensions}
\label{sec:products-extensions} 
        
We begin with a general result on order sequences and domination.
        
\begin{proposition}
Let $G$ and $H$ be groups of the same order. Then $\os(G)$ dominates $\os(H)$
if and only if there is a bijecton $f:G\to H$ such that $o(g)\ge o(f(g))$ for
all $g\in G$.
\label{p:bijection}
\end{proposition}
       
\begin{proof} The forward implication is clear. For the converse, suppose
that the bijection exists, and let $G=\{g_1,\ldots,g_n\}$ where
$o(g_1)\le\cdots\le o(g_n)$. Then, for $1\le k\le n$, there are at least $k$
elements of $H$ whose orders do not exceed $o(g_k)$, namely $f(g_1),\ldots,
f(g_k)$; so, if $h_k$ is the element of $k$ whose order is the $k$th (in
increasing order), then $o(h_k)\le o(g_k)$ as required.
\end{proof}
        
On the basis of this, we make a stronger definition. For groups $G$ and $H$
of the same order, we say that $\os(G)$ \emph{strongly dominates} $\os(H)$
if there is a bijection $f:G\to H$ such that $o(f(g))\mid o(g)$ for all
$g\in G$.

With this terminology, we can state Amiri's result~\cite{amiri-strongdom}:

\begin{theorem}
The order sequence of $\Z_n$ strongly dominates that of any other group of
order~$n$.
\label{t:amiristrong}
\end{theorem}

We use strong domination in the proof of a theorem later in the
paper (Theorem~\ref{thm:non-nil-min-os-existence}). But it might warrant
further investigation. It is not equivalent to domination as previously
defined. The group $\Z_3:\Z_4$ of order $12$, where the generator of $\Z_4$
conjugates the generator of $\Z_3$ to its inverse, dominates the alternating
group $A_4$, but does not strongly dominate it: their order sequences are
respectively
\begin{eqnarray*}
&&(1, 2, 3, 3, 4, 4, 4, 4, 4, 4, 6, 6)\\
&\hbox{and}&(1, 2, 2, 2, 3, 3, 3, 3, 3, 3, 3, 3)
\end{eqnarray*}
(these are \texttt{SmallGroup(12,$i$)} for $i=1$ and $i=3$ in the
\textsf{GAP} library~\cite{gap}.)
        
We define two operations on non-decreasing sequences $x$ and $y$. If these
sequences have lengths $m$ and $n$ respectively, then $xy$ is the sequence
formed by taking all products $x_iy_j$ for $1\le i\le m$ and $1\le j\le n$ and
writing them in non-decreasing order. Similarly, $x\vee y$ is obtained by
taking all numbers $\lcm\{x_i,y_j\}$ for $1\le i\le m$ and $1\le j\le n$
and writing them in non-decreasing order.
        
\begin{proposition}
\begin{enumerate}
\item
For two non-decreasing sequences $x$ and $y$ of lengths $m$ and $n$
respectively, $xy=x\vee y$ if and only if $x_i$ and $y_j$ are coprime for
$1\le i\le m$ and $1\le j\le n$.
\item
Suppose that $x_i\mid m$ for all $i$, and $y_j\mid n$ for all $j$, where
$\gcd(m,n)=1$. Suppose further that each of the sequences $x$ and $y$ contains
a unique term equal to $1$. Then the product sequence $xy$ and the numbers $m$
and $n$ uniquely determine $x$ and $y$.
\end{enumerate}
\end{proposition}
        
\begin{proof}
The proof of (a) is obvious. For (b), we see that $x$ consists of the terms of
$xy$ coprime to $n$, and $y$ consists of the terms coprime to $m$.
\end{proof}
        
\begin{theorem}\label{t:sequence-product}
\begin{enumerate}
\item For any two groups $G$ and $H$, $\os(G\times H)=\os(G)\vee\os(H)$. In
particular, $\os(G\times H)=\os(G)\os(H)$ if and only if $|G|$ and $|H|$ are
coprime.
\item
Assume that $\gcd(|G|,|H|)=1$. Given $\os(G\times H)$ and the numbers $|G|$
and $|H|$, the sequences $\os(G)$ and $\os(H)$ are determined.
\item If $|G|$ and $|H|$ are not coprime, then there is no group $K$ for
which $\os(G)\os(H)=\os(K)$.
\end{enumerate}
\end{theorem}
        
\begin{proof} (a) This  follows since $o((g,h))=\lcm\{o(g),o(h)\}$.

\medskip

(b) This is immediate from the preceding Proposition.
        
\medskip
        
(c) Suppose that $|G|$ and $|H|$ are not coprime, and let $p$ be a prime
number dividing both $|G|$ and $|H|$. The number of elements of order $p$ 
in $G$ is congruent to $-1$ (mod~$p$). (This follows from the proof of
Cauchy's Theorem that $G$ contains elements of order~$p$.) Similarly for $H$.
So the number of terms in $\os(G)\os(H)$ equal to $p$ is congruent to $-2$
(mod~$p$). Hence $\os(G)\os(H)$ cannot be the order sequence of a group.
\end{proof}

   The following corollary is immediate from Theorem \ref{t:sequence-product}. 
   
   \begin{corollary}
   	\label{c:seq-product}
   	Let $G,H$ be two finite groups with $\text{gcd}(|G|,|H|)=1$. Then, $\rho(G\times H)= \rho(G)^{|H|} \rho(H)^{|G|}.$ 	
   \end{corollary}
        
\begin{theorem}\label{t:extension}
Let $G$ and $H$ be groups of coprime order, and suppose that $G$ is
abelian. Let $K$ be any extension of $G$ by $H$. Then
$\os(G\times H)=\os(G)\os(H)$ strongly dominates $\os(K)$.
\end{theorem}
        
\begin{proof} According to the Schur--Zassenhaus Theorem, $G$ has a
complement in $K$; that is, there is a subgroup of $K$, which we identify
with $H$, such that $G\cap H=\{1\}$ and $GH=K$. (Since $G$ is abelian, only
Schur's part of the proof is required.) Thus every element of $K$
is uniquely written as $gh$ for $g\in G$ and $h\in H$. So the map 
$f:G\times H\to K$ defined by $f((g,h))=gh$ is a bijection. By
Proposition~\ref{p:bijection}, it suffices to show that $o(g)o(h)\ge o(gh)$.
        
Suppose that $o(g)=m$ and $o(h)=n$. We will prove that $(gh)^{mn}=1$, from
which the result follows.
        
We have
\[(gh)^n=g.hgh^{-1}.h^2gh^{-2}\cdots h^{n-1}gh^{-(n-1)}h^n,\]
and we have $h^n=1$. Since $G$ is a normal subgroup of $K$, the elements
$g$, $hgh^{-1}$, \dots, $h^{n-1}gh^{-(n-1)}$ all belong to $G$, and all have 
order $m$, since they are conjugate to $g$. In an abelian group, the product
of elements of order $m$ has order dividing $m$; so $((hg)^n)^m=1$, and the
proof is complete.
\end{proof} 
        
Now we give a couple of results on direct products.

\begin{proposition}\label{p:products}
Suppose that $G_1$, $G_2$ and $H$ are groups such that $|G_1|=|G_2|$ and
$\os(G_1)$ dominates $\os(G_2)$. Suppose that either
\begin{enumerate}
\item $|G_1|$ and $|H|$ are coprime; or
\item $\os(G_1)$ strongly dominates $\os(G_2)$.
\end{enumerate}
Then $\os(G_1\times H)$ dominates $\os(G_2\times H)$.
\end{proposition}
    
\begin{proof}
Let $f$ be a bijection from $G_1$ to $G_2$ satisfying $o(g)\ge o(f(g))$ (or,
in case (b), $o(f(g))\mid o(g)$). Then define a bijection
$f':G_1\times H\to G_2\times H$ by the rule that $f'((g,h))=(f(g),h)$.
\begin{enumerate}
\item If $|G_1|$ and $|H|$ are coprime, then
\[o((g,h))=o(g)o(h)\ge o(f(g))o(h)=o((f(g),h)),\]
so by Proposition~\ref{p:bijection}, $\os(G_1\times H)$ dominates
$\os(G_2\times H)$.
\item If $f$ satisfies the conditions for strong domination, then
\[o((f(g),h))=\lcm(o(f(g)),o(h))\mid\lcm(o(g),o(h))=o((g,h)),\]
and again Proposition~\ref{p:bijection} gives the result -- indeed we conclude
that the domination is strong.
\end{enumerate}
\end{proof}

From this we can prove a two-sided version:

\begin{proposition}
Let $G_1$ and $G_2$ be groups of the same order, and let $H_1$ and $H_2$ be
groups of the same order. Suppose that $\os(G_1)$ dominates $\os(G_2)$ and
$\os(H_1)$ dominates $\os(H_2)$. Moreover, suppose that one of the following
holds:
\begin{enumerate}
\item $|G_1|$ and $|H_1|$ are coprime;
\item $\os(G_1)$ strongly dominates $\os(G_2)$ and $\os(H_1)$ strongly
dominates $\os(H_2)$.
\end{enumerate}
Then $\os(G_1\times H_1)$ dominates $\os(G_2\times H_2)$.
\end{proposition}

This follows immediately from two applications of the preceding result.

\section{Abelian groups and the partition lattice}
\label{sec:abelian-groups} 

We start this section by describing the order sequence of a finite abelian $p$-group. Let 
$|G|=p^n$ where $p$ is prime and $G$ is abelian. By the Fundamental Theorem of
Abelian Groups, 
$$G\cong \Z_{p^{r_1}} \times \Z_{p^{r_2}} \times  \dots \times \Z_{p^{r_k}},$$
where $r_i \geq 1$ for $1 \leq i \leq k$, and $r_1\geq r_2\geq\dots \geq r_k$.
Then $r_1+r_2+\cdots+r_k=n$, so $(r_1,r_2,\ldots,r_k)$ is a partition of $n$.
(We examine partitions further below.) Define $s_1,s_2,\ldots,s_\ell$, where
$\ell=r_1$, by the rule that
\[s_j=|\{i:r_i\ge j\}|.\] 
Thus, for example, $s_1=k$.

We note that the numbers $s_j$ in turn determine the numbers $r_i$: we will see
the exact relationship shortly.

\begin{proposition}\label{p:orders}
With the above notation, the number of elements of order dividing $p^j$ in $G$
is $p^{s_1+\cdots+s_j}$.
\end{proposition}

\begin{proof}
The elements with order dividing $p^j$ form a subgroup $A_j$ of $A$. Now
$A_j\le A_{j+1}$, and $A_{j+1}/A_j$ consists of the elements of order 
$1$ or $p$ in $A/A_j$; this group is generated by the cosets containing
elements of order $p^{j+1}$, which has rank $s_{j+1}$, and so its cardinality
is $p^{s_{j+1}}$. So $|A_{j+1}|=p^{s_{j+1}}|A_j|$. Since $A_0$ is the identity
group, induction now completes the proof of the Proposition.
\end{proof}

Using this, we immediately have the following corollary which is also proved in \cite{ronald}.

\begin{theorem}\label{t:abelian}
Two finite abelian groups have the same order sequence if and only if they are
isomorphic.
\end{theorem}

\begin{proof}
The reverse implication is clear. For the forward implication, it suffices to
prove the result for groups of prime power order. By Proposition~\ref{p:orders},
the order sequence determines the numbers $s_1,\ldots,s_{\ell}$, and hence the
numbers $r_i$, and hence the isomorphism type of the group.
\end{proof}

As we have already seen, we cannot expect such a result for a wider class of
groups containing the class of abelian groups; even for $p$-groups with $p$ an
odd prime, it is not true.

We can also count the numbers of elements, or cyclic subgroups, of given order:
\begin{proposition}
With the notation introduced before Proposition~\ref{p:orders},
\begin{enumerate}
\item the number of elements of order $p^j$ is
$p^{s_1+\cdots+s_{j-1}}(p^{s_j}-1)$;
\item the number of cyclic subgroups of order $p^j$ is
$p^{s_1+\cdots+s_{j-1}-j+1}(p^{s_j}-1)/(p-1)$;
\item the total number of cyclic subgroups is
\[1+p^{s_1-1}+\cdots+p^{s_1+\cdots+s_{m-1}-m+1}+(p^{n-m+1}-1)/(p-1)\]
where $|G|=p^n$ and the exponent of $G$ is $p^m$.
\end{enumerate}
\end{proposition}

\begin{proof}
(a) By Proposition~\ref{p:orders}, this number is
$p^{s_1+\cdots+s_j}-p^{s_1+\cdots+s_{j-1}}$.

(b) Each cyclic subgroup of order $p^j$ has $p^{j-1}(p-1)$ generators.

(c) This is obtained by summing the formulae in (b), noting that 
the maximum $j$ for which $s_j>0$ is $m$ and $s_1+\cdots+s_m=n$.
\end{proof}

\medskip

Now we relate the preceding analysis of abelian $p$-groups to the theory of
partitions of integers, referring to \cite[Section 1.1]{macdonald}.

A \emph{partition} of $n$ is a non-decreasing sequence $(r_1,\ldots,r_k)$ of
positive integers with sum $n$. Given two partitions $a=(r_1,\ldots,r_k)$
and $b=(s_1,\ldots,s_\ell)$ of $n$, to compare them we append zeros to the
shorter sequence if necessary to make them have the same length; then we
write $a\succeq b$ if, for all relevant $j$, we have
\[r_1+\cdots+r_j\ge s_1+\cdots+s_j.\]
Thus, from the top, the order begins 
\[(n)\succeq(n-1,1)\succeq(n-2,2)\succeq(n-2,1,1).\]
This is called the \emph{natural partial order} on partitions (also called
\emph{majorization}). It is not a 
total order; for example, $(2,2,2)$ and $(3,1,1,1)$ are incomparable.

A parttion $(r_1,\ldots,r_k)$ of $n$ can be represented by a \emph{Young	diagram} or \emph{Ferrers diagram}, made up of $n$ squares in the plane, in 
left-aligned rows of lengths $r_1,\ldots,r_k$. The picture shows the partitions
$(4,2,1)$ and $(3,2,1,1)$.

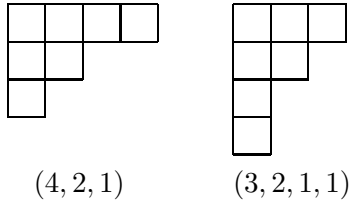
\begin{figure}[htbp]
\begin{center}
\setlength{\unitlength}{0.5mm}
\begin{picture}(90,50)
\multiput(0,50)(0,-10){2}{\line(1,0){40}}
\put(0,30){\line(1,0){20}}
\put(0,20){\line(1,0){10}}
\multiput(0,50)(10,0){2}{\line(0,-1){30}}
\put(20,50){\line(0,-1){20}}
\multiput(30,50)(10,0){2}{\line(0,-1){10}}
\put(7,0){$(4,2,1)$}
\multiput(60,50)(0,-10){2}{\line(1,0){30}}
\put(60,30){\line(1,0){20}}
\multiput(60,20)(0,-10){2}{\line(1,0){10}}
\multiput(60,50)(10,0){2}{\line(0,-1){40}}
\put(80,50){\line(0,-1){20}}
\put(90,50){\line(0,-1){10}}
\put(60,0){$(3,2,1,1)$}
\end{picture}
\end{center}
\caption{\label{f:young}Young diagrams}
\end{figure}

The \emph{conjugate} $a'$ of the partition $a=(r_1,\ldots,r_k)$ is the parttion
$b=(s_1,\ldots,s_\ell)$, where 
\[s_i=|\{j:r_j\ge i\}|.\]
This corresponds simply to reflecting the Young diagram in the diagonal. So
$a''=a$ for any partition $a$. The two partitions  in the figure are conjugates
of each other.

An abelian group $A$ of order $p^n$ has a unique expression of the form
\[A\cong\Z_{p^{r_1}}\times\cdots\times\Z_{p^{r_k}},\]
where $a=(r_1,\ldots,r_k)$ is a partition of $n$, which we will call the
\emph{defining partition} of $A$.

Our arguments in the proof of Theorem~\ref{t:abelian} show that the partition
$(s_1,s_2,\ldots,s_{\ell})$ used in the proof of that theorem is conjugate
to the defining partition of $A$. Now the fact that $a''=a$ shows that the
construction of $(r_1,\ldots,r_k)$ from $(s_1,\ldots,s_l)$ is formally identical
to the construction in the other direction.

\begin{theorem}
\label{thm:relation-with-poset-partitionlattice}
Let $A$ and $C$ be two abelian groups of order $p^n$, with defining partitions
$a$ and $c$ respectively. Then the following are equivalent:
\begin{enumerate}
\item $\os(A)$ dominates $\os(C)$;
\item $c'\succeq a'$;
\item $a\succeq c$.
\end{enumerate}
\end{theorem}

\begin{proof} $\os(A)$ dominates $\os(C)$ if and only if, for each $i$,
the number of elements of order dividing $p^i$ is at least as large in $C$
as in $A$. By Proposition~\ref{p:orders} and the definition of the natural
partial order, the truth of this for all $i$ is equivalent to $c'\succeq a'$.
Thus (a) and (b) are equivalent. The equivalence of (b) and (c) is standard;
see \cite[(1.11)]{macdonald}.
\end{proof} 

The poset of partitions of an integer has been the subject of research covering
several areas of mathematics. See~\cite{brylawski,greene} for some of the
results. We can regard the poset of order sequences of groups as a kind of
non-commutative generalisation of this. We return briefly to this topic later.

\paragraph{Remark} The paper~\cite{greene} gives estimates for the lengths of
maximal chains in the partition lattice. It is very easy to see that the 
maximum chain length of the lattice of partitions of $n$ tends to infinity
with $n$. Now any finite poset is embeddable in a product of finite chains
(for example, take all linear extensions of the given poset). So we conclude:

\begin{proposition}
For any finite poset $P$, there exists $n$ such that $P$ is embeddable in the
poset of order sequences of abelian groups of order $n$.
\end{proposition}

We are next interested in the comparison between the number of cyclic subgroups of two abelian $p$-groups of order $p^n$. For a group $G$, let $\cyc(G)$ denote the number of cyclic subgroups of $G$. The following result is well-known.

\begin{lemma}
\label{lem:majorization-removable-boxes}
Let $a=(r_1, r_2, \dots, r_k)$ and $c=(s_1, s_2, \dots, s_{\ell})$ be two partitions of $n$. Then $a \succeq c$ if and only if the Young diagram for $c$ can be obtained from that of $a$ by successively moving boxes from a higher row to a lower row (one box at a time), in such a way that each intermediate step is the Young diagram of a partition of $n$. 	
\end{lemma}

\begin{lemma}
\label{lem:comparison-cyc-sub-onestep}
Let $A, A'$ and $A''$ be three abelian groups of order $p^n$ with defining partitions $a=(r_1, r_2, \dots, r_k), a'=(r_1, r_2, \dots, r_{j_1}-1, \dots, r_{j_2}+1, \dots, r_k),$ and $ a''= 
(r_1, r_2, \dots, r_{j_1}-1, \dots, r_k, 1)$ respectively. Then, $\cyc(A) < \cyc(A')$ and $\cyc(A) < \cyc(A'')$.
\end{lemma}

\begin{proof}
It is clear that \[\cyc(A)= \prod _{i=1}^k (1+r_i), \hspace{3 mm} \text{ and } \hspace{3 mm} \cyc(A')= r_{j_1} (2+r_{j_2}) \prod _{i=1, i \neq j_1, j_2 }^k (1+r_i).\] As $a'$ is a valid partition of $n$, we must have $r_{j_1} >r_{j_2}+1$ and therefore we have $2r_{j_1}+r_{j_1}r_{j_2} > 1+ r_{j_1}+r_{j_2}+r_{j_1}r_{j_2}$. Hence,
 $\cyc(A') > \cyc(A)$.  
 
 As $a''$ is a valid partition of $n$, we must have $r_{j_1} \geq 2$ and therefore $\cyc(A'') > \cyc(A)$.
This completes the proof.
\end{proof}

We are now in a position to prove the next result which tells about the connection between the order sequence and the number of cyclic subgroups of $p$-groups. 

\begin{theorem}
\label{thm:order-seq-pgroup-number-cyclic-subgrp}
Let $A$ and $C$ be two abelian groups of order $p^n$, with defining partitions
$a$ and $c$ respectively. If $\os(A)$ dominates $\os(C)$, then $\cyc(A) \leq \cyc(C)$. 
\end{theorem}

\begin{proof}
Let \[A\cong\Z_{p^{r_1}}\times\cdots\times\Z_{p^{r_k}},\]
and 
\[C \cong\Z_{p^{s_1}}\times\cdots\times\Z_{p^{s_{\ell}}},\]
where $a=(r_1,\ldots,r_k)$ and $b=(s_1, s_2, \dots, s_{\ell})$ are defining partitions of $A$ and $C$ respectively. Then, the number of cyclic subgroups of $A$ and $C$ are respectively 
\[\cyc(A)=(r_1+1)(r_2+1)\dots(r_k+1)\]
and \[\cyc(C)=(s_1+1)(s_2+1)\dots(s_{\ell}+1).\]
By Theorem \ref{thm:relation-with-poset-partitionlattice}, we have $a\succeq c$. By Lemma \ref{lem:majorization-removable-boxes}, the Young diagram for $c$ can be obtained from that of $a$ by successively moving boxes from a higher row to a lower row (one box at a time), in such a way that each intermediate step is the Young diagram of a partition of $n$. Let us denote the partitions of the intermediate steps by $b^1, b^2, \dots, b^r$ and we also denote $a$ by $b^0$ and $c$ by $b^{r+1}$. Then we have 
\[a (=b^0) \succeq b^1 \succeq b^2 \succeq \dots \succeq \dots \succeq b^r \succeq c (=b^{r+1})\]
where the Yound diagram for each $b^i$ is obtained from that of $b^{i-1}$ by moving exactly one box from a higher row to a lower row. If $B^i$ denotes the abelian $p$-group of order $p^n$ with defining partition $b^i$, by using Lemma \ref{lem:comparison-cyc-sub-onestep}, we now have 
 \[\cyc(A) \leq  \cyc(B^1) \leq \cyc(B^2) \leq \dots \leq \cyc(B^r) \leq \cyc(C).\]
 This completes the proof. 
\end{proof}

It is clear that the converse of Theorem \ref{thm:order-seq-pgroup-number-cyclic-subgrp} need not hold, in general. For example, we can consider the groups $\Z_{16} \times \Z_2 \times \Z_2$ and $\Z_8 \times \Z_8$. Then, $\cyc(\Z_{16} \times \Z_2 \times \Z_2)=20$ and $\cyc(\Z_8 \times \Z_8)=16$ but the order sequences of the groups are not comparable. 

\section{Nilpotent and non-nilpotent groups} 
\label{sec:nilpotent-nonnilpotentgroups}
     
In this section, we are primarily interested in studying the minimality of order sequence among finite nilpotent groups. 

It is easy to see that the order sequence of any finite $p$-group of order $p^r$ dominates the order sequence of the elementary abelian $p$-group $E(p^r)$.

\begin{theorem}
\label{thm:nilp-order-seq-extreme-cases}
Let $G=P_1 \times P_2 \times \dots \times P_k$ be a finite nilpotent group. If the order sequence of $G$ is minimal among nilpotent groups, then for each $1 \leq i \leq k$, $P_i$ is a group of prime exponent.  
\end{theorem} 

\begin{proof}
We prove by induction on the number of primes $k$. For $k=1$ this is of course true. For a nilpotent group $G$, it can be written as 
$G= P_1 \times P_2 \times \dots \times P_k$, where $P_i$ are the Sylow subgroups.  Let $G_1 = P_1 \times P_2 \times \dots \times P_{k-1} $ and $G_2 = P_k$. Then $\os(G_1)$ is minimal among nilpotent groups of its order. For if not, then we could replace it with a nilpotent group $G_1^*$ yielding a smaller sequence, whence $\os(G)$ would dominate $\os(G_1^*\times P_k)$ by Proposition~\ref{p:products}. The induction hypothesis now shows that $P_i$ is a group of prime exponent for $1\le i\le k-1$.

Similarly, the order sequence of $P_k$ is minimal, so $P_k$ is a group of prime exponent.   This completes the proof of the Theorem. 
\end{proof} 

We have already seen that if a group $G$ has the same order sequence as a cyclic group, then $G$ must be cyclic; but it may happen that $G$ is non-abelian but $G$ has the same order sequence as an abelian group. For example, one can take a non-abelian group of order $p^3$ with every element of order $p$. We next ask the same question for nilpotent groups. We show first that nilpotency is
determined by the order sequence of the group.

\begin{theorem}
If $G$ is a nilpotent group, and $H$ is a group of the same order such that
$\os(G)=\os(H)$, then $H$ is nilpotent.
\end{theorem}

\begin{proof} Let $|G|=|H|=p_1^{a_1}\cdots p_r^{a_r}$. All elements of
$p_i$-power order in $G$ belong to the unique Sylow $p_i$-subgroup of $G$,
and there are $p_i^{a_i}$ of them. By assumption, $H$ also has exactly 
$p_i^{a_i}$ elements of $p_i$-power order, and it has a Sylow $p_i$-subgroup
$P_i$; thus $P_i$ contains all elements of $p_i$-power order, and so is the
unique Sylow $p_i$-subgroup, and is normal in $H$. Since this holds for all $i$,
all the Sylow subgroups of $H$ are normal, and $H$ is nilpotent.
\end{proof}

The proof shows that, if $G$ is nilpotent and $\os(G)=\os(H)$, then for each prime
$p_i$ dividing $|G|$, the Sylow $p_i$-subgroups of $G$ and $H$ have the same
order sequence.

\medskip

Next we show that, if there is a non-nilpotent group of order $n$, then
nilpotent groups cannot realise order sequences which are minimal under
domination. 

First we require a lemma.

\begin{lemma}\label{l:order-non-nilp}
For a positive integer $n$, the following are equivalent:
\begin{itemize}
\item there exists a non-nilpotent group of order $n$;
\item $n$ is divisible by $p^dq$, where $p$ and $q$ are primes and $q\mid p^d-1$.
\end{itemize}
\end{lemma} 

\begin{proof}
Suppose first that $G$ is a non-nilpotent group of order $n$. Then $G$ contains a minimal non-nilpotent subgroup (one all of whose proper subgroups are nilpotent). So it suffices to deal with the case where $G$ is minimal non-nilpotent.

The minimal non-nilpotent groups were determined by Schmidt~\cite{schmidt}. A convenient reference is \cite{ber}, which we use here. Page 3456 of this paper gives a list of five types of group, and Theorem~3 asserts that the Schmidt groups 
(the minimal non-nilpotent groups) are exactly those of types II, IV and V.
\begin{itemize}
\item Type II groups are semidirect products $[P]Q$, where $Q$ is cyclic of order $q^r$, so that the $p$-group $P$ is a faithful irreducible $Q/Q^q$-module with trivial centralizer. Thus the cyclic group $Q/Q^q$ of order $q$ acts fixed-point-freely on $P\setminus\{1\}$, so $q\mid|P|-1$.
\item Type IV groups are semidirect products $[P]Q$, where $P$ is special of rank $2m$ and the cyclic group $Q$ has a faithful irreducible action on $P/\Phi(P)$. Thus $q$ divides $|P/\Phi(P)|=p^{2m}$.
\item Type V groups are semidirect products $[P]Q$, where $|P|=p$, $Q$ is cyclic of order $q^r$, and $Q$ induces an automorphism group of $P$ of order $q$; so $q\mid p-1$.
\end{itemize}

Conversely, suppose that $p$ and $q$ are primes such that $q\mid p^d-1$. The group $B=\{x\mapsto ax+b\}$ of permutations of the finite field of order $p^d$, where $a$ runs through the $q$th roots of unity in the field and $b$ runs through the whole field, is a non-nilpotent group of order $p^dq$. If $p^dq\mid n$, then let $K$ be any abelian group of order $n/(p^dq)$; then $B\times K$ is a non-nilpotent group of order $n$.
\end{proof}

\begin{theorem}
	\label{thm:non-nil-min-os-existence}
Let $n$ be a positive integer for which there exists a non-nilpotent group of
order $n$. Then there is a non-nilpotent group $H$ of order $n$ such that
the order sequence of any nilpotent group $G$ of order $n$ properly dominates
$\os(H)$.
\end{theorem}

\begin{proof}
By Theorem~\ref{thm:nilp-order-seq-extreme-cases}, the abelian group $G$ of
order $n$ whose Sylow subgroups have prime exponent is dominated by every
nilpotent group of order $n$.

By Lemma~\ref{l:order-non-nilp}, there is a non-nilpotent group $B$ of order
$p^dq$ dividing $n$. Let $A$ be the direct product of an elementary abelian
group of order $p^d$ and a cyclic group of order $q$. Now $B$ has the same
number of elements of order $1$ or $p$ as $A$ does; the remaining elements of
$B$ all have order $q$, while $A$ has elements of orders $q$ and $pq$. Let $f$
be a bijection from $A$ to $B$ mapping the identity to the identity,
elements of order $p$ to elements of order $p$, and the remaining elements
arbitrarily. This bijection shows that $\os(A)$ strongly dominates $\os(B)$.

Let $K$ be abelian group of order $n/(p^dq)$, all of whose Sylow subgroups have
prime exponent. Then $G\cong A\times K$. By Proposition~\ref{p:products}(b),
$H=B\times K$ is a non-nilpotent group of order $n$ dominated by $G$, and
hence by every nilpotent group of order $n$. This completes the proof.
\end{proof}

        \section{Bounds on several functions}
        \label{sec:max-order-seq}
        
In this section we use the maximality
        of $\Z_n$ under (strong) domination to establish bounds on the gaps between the values
        of $\psi$ and of $\rho$ on cyclic and non-cyclic groups, and characterize
        groups meeting these bounds (Theorem~\ref{thm:upper-bdd-prod-eltorder-G}).
        
        \begin{corollary}
        	Let $F$ be a symmetric function of $n$ arguments which is a strictly
        	increasing function of each argument and $G$ be a non-cyclic group of order~$n$.
        	Then $F(\os(G))<F(\os(\Z_n))$.
        \end{corollary}
        
        Theorem \ref{thm:amiri-cia} is an immediate consequence. 
        We can also establish a gap between the cyclic group and other
        groups of order $n$.
        
         Domenico, Monetta and Noce \cite{Domenico-Monetta-Noce} proved the following upper bound for the function $\rho(G)$ where $G$ has a Sylow tower. We recall that $G$ is said to admit a Sylow tower if there exists a normal series
        $$\{e\}=G_0 \leq G_1 \leq G_2 \leq  \dots \leq G_n=G$$
        such that $G_{i+1}/G_i$ is isomorphic to a Sylow subgroup of $G$ 
        for every $0 \leq i \leq n-1.$ 
        
        \begin{theorem} \label{thm:dom-mos-nose-prod-elt-order}
        	Let $G$ be a non-cyclic group of order $n$ admitting a Sylow tower. Then 
        	$$\rho(G) \leq q^{-q} \rho(\mathbb{Z}_n),$$
        	where $q$ is the smallest prime dividing $n$. The same inequality holds if $n=p^aq^b$ with $p>q$ or $G$ is a Frobenius group.  
        \end{theorem}   
        
        The cited paper also includes similar results for other classes such as
        super-solvable groups.
        
        In this paper, we improve their result
        in three ways. First, we do not need to assume the existence of a Sylow tower,
        or any restriction on the group $G$.
        Second, our bound is stronger. Third, we can deal with other functions; we give
        a result for the sum of the orders of group elements as an example, and
        determine the groups meeting our bound for either sum or product.
        
        \begin{theorem} \label{thm:upper-bdd-prod-eltorder-G}
         Let $G$ be a non-cyclic  group of order $n.$  Then
        	\begin{enumerate}
        		\item
        		$\rho(G) \leq q^{-\phi(n)} \rho (\mathbb{Z}_n)$;
        		\item  
        		$\psi(G) \leq \psi(\Z_n)-n\phi(n)(q-1)/q$;
        	\end{enumerate}
        	where $q$ is the smallest prime divisor of $n$ and $\phi$ is Euler's functon.
        	Equality holds if and only if either $G = \mathbb{Z}_q \times \Z_q$ or $G$ is
        	the quaternion group of order~$8$.
        \end{theorem}  
        
        \begin{proof}
        	We compare the order sequences of $G$ and $\mathbb{Z}_n$. Let
        	$$\os(G)=(o(g_1), o(g_2), \dots, o(g_n)) \hbox{ and }
        	\os(\Z_n)=(o(a_1), o(a_2), \dots, o(a_n)).$$
By Theorem~\ref{t:amiristrong}, we have $o(g_i) \leq o(a_i)$ for all $1 \leq i \leq n.$ Moreover, $\Z_n$ has $\phi(n)$ elements of order $n$ and the last $\phi(n)$ terms of $\os(G)$ are clearly $\leq n/q$. So replacing the $\phi(n)$
        	elements of $\os(\Z_n)$ equal to $n$ by $n/q$ gives a sequence $a$
        	which still dominates $\os(G)$. (This is still a non-decreasing sequence since
        	no element of $G$ has order larger than $n/q$.) Therefore
        	$$\rho(G) \leq q^{-\phi(n)} \rho (\mathbb{Z}_n)
        	\hbox{ and }
        	\psi(G)\leq\psi(\Z_n)-\phi(n)(n-n/q)$$
        	(the values on the right are the sum and product of elements of $a$).
        	
        	We leave to the reader the analogous result for $\psi_k$.
        	
        	It is easy to see that when $G=\Z_q\times\Z_q$ or $G=Q_8$, the equality holds. 
        	
        	Conversely, suppose that $G$ is any group attaining either of the bounds. The
        	proof of the inequality shows that $\os(G)$ is obtained from $\os(\Z_n)$ by
        	replacing the $\phi(n)$ orders $n$ by orders $n/q$, leaving the others as
        	before. This implies that, if $m$ is any divisor of $n$ other than $n$ and
        	$n/q$, then $G$ contains the same number of elements of order $m$ as $\Z_n$,
        	namely $\phi(m)$ of them; so $G$ has a unique cyclic subgroup of order $m$,
        	which is normal in $G$ and contains all elements of order $m$.
        	
        	We separate into three cases. Suppose first that $n$ has at least three prime
        	divisors, say $q$ (the smallest), $r$ and $s$. Then $G$ has unique cyclic
        	subgroups of orders $n/r$ and $n/s$; their product is a cyclic subgroup of
        	order~$n$, neccessarily equal to~$G$, a contradiction.
        	
        	Next suppose that $G$ is a $q$-group. Since it is not cyclic, either
        	$G=\Z_q\times\Z_q$, or $|G|>q^3$. In the second case, $G$ contains a unique
        	subgroup of order $q$. A theorem  of Burnside (see~\cite[Theorem 12.5.2]{hall})
        	shows that $G$ is cyclic or
        	generalized quaternion. Cyclic groups are excluded, and it is easy to see that
        	the only generalized quaternion group satisfying the conditions is $Q_8$.
        	
        	Finally suppose that only two primes $q$ and $r$ divide $|G|$. By our earlier
        	remark, the Sylow $q$-subgroup of $G$ is cyclic and normal. Let $R$ be the
        	Sylow $r$-subgroup. If $|Q|>q$, then also $R$ is cyclic and normal, whence $G$
        	is cyclic, a contradiction. So $|Q|=q$. Let $R$ be the Sylow $r$-subgroup.
        	The number of conjugates of $R$ is $1$ or $q$, and is congruent to $1$ mod~$r$;
        	so $R$ is normal in $G$, and $G=Q\times R$. Now $R$ must be non-cyclic; since
        	$r$ is odd, Burnside's theorem implies that $G$ contains more than $r-1$
        	elements of order~$r$, a contradiction.
        \end{proof}

        It is easy to show that, for any $n \neq q $, we have $\phi(n)\geq q$, and therefore the bound in Theorem \ref{thm:upper-bdd-prod-eltorder-G} is stronger than the bound in Theorem \ref{thm:dom-mos-nose-prod-elt-order}.

        Domenico, Monetta and Noce \cite[Proposition 8]{Domenico-Monetta-Noce} also proved the following upper bound for $\rho(G)$ where $G$ is nilpotent. 
        
        \begin{theorem} \label{thm:dom-mos-nose-prod-elt-order-nilpotent}
        	Let $G$ be a non-cyclic nilpotent group of order $n$. Then 
        	$$\rho(G) \leq q^{-\frac{n(q-1)}{q}} \rho(\mathbb{Z}_n),$$
        	where $q$ is the smallest prime dividing $n$.
        \end{theorem}   
        
        In the following result, we improve the theorem. 
        
        \begin{theorem}
        	\label{thm:dom-mos-nose-prod-elt-order-nilpotent-improve}
        	Let $G=\Z_m \times G_1$ where $G_1$ is a nilpotent group and moreover $G_1 = P_1 \times P_2 \times \dots \times P_r$ where every $P_i$ is a noncyclic $p$-group (w.r.t the prime $p_i$) and $\text{gcd}(m,|G_1|)=1$. Then, 
        	$$\rho(G) \leq \left( \prod_{i=1}^r p_i^{-\frac{(p_i-1)}{p_i}} \right)^{|G|} \rho(\mathbb{Z}_{|G|}).$$
        	Moreover, equality holds when each $P_i$ is $ \mathbb{Z}_{p_i} \times \Z_{p_i}.$  
        \end{theorem}
        
        \begin{proof}
        	We prove this by induction on $r$. When $r=1$, we have $G= \Z_m \times P_1$.  By Corollary \ref{c:seq-product}, we clearly have 
        	$$\rho(G)=\rho(\Z_m \times P_1)= \rho(\Z_m)^{|P_1|} \rho(P_1)^{m}. $$
        	As in Theorem \ref{thm:upper-bdd-prod-eltorder-G}, we compare the order sequences of $P_1$ and $\Z_{|P_1|} $. We know that the order sequence of $P_1$ is dominated by the order sequence of  $\Z_{|P_1|} $ and
        	moreover, $\Z_{|P_1|}$ has $\frac{|P_1|(p_1-1)}{p_1}$ elements of order $|P_1|$ and the last $\frac{|P_1|(p_1-1)}{p_1}$ terms of $\os(P_1)$ are clearly $\leq \frac{|P_1|}{p_1}$. Therefore,
        	we clearly have 
        	
        	\begin{eqnarray*} \rho(\Z_m)^{|P_1|} \rho(P_1)^{m} & \leq & \rho(\Z_m)^{|P_1|}  \left( p_1^{-\frac{|P_1|(p_1-1)}{p_1}} \rho(\Z_{|P_1|}) \right)^{m} \\
        		& 	= &  \left( p_1^{-\frac{(p_1-1)}{p_1}} \right)^{m|P_1|} \rho(\Z_m)^{|P_1|} \rho(\Z_{|P_1|}) ^{m} \\ & = & \left( p_1^{-\frac{(p_1-1)}{p_1}} \right)^{m|P_1|} \rho(\Z_{m|P_1|}) .
        	\end{eqnarray*} 
        	Thus, when $r=1$, the statement holds. We assume the statement for $r=\ell-1$ and prove for $r=\ell$.
        	Let $G=\Z_m \times P_1 \times \dots \times P_{\ell-1} \times P_{\ell}.$ Moreover, let $H= \Z_m \times P_1 \times \dots \times P_{\ell-1} $.  By Corollary \ref{c:seq-product}, we clearly have 
        	
        	$$\rho(G)=\rho(H \times P_{\ell})= \rho(H)^{|P_{\ell}|} \rho(P_{\ell})^{|H|}. $$
        	By induction and using the fact that $P_{\ell}$ is non-cyclic nilpotent, we have 
        	
        	\begin{align*} \rho(H)^{|P_{\ell}|} \rho(P_{\ell})^{|H|} &  \leq  \left ( \left( \prod_{i=1}^{\ell-1} p_i^{-\frac{(p_i-1)}{p_i}} \right)^{|H|} \rho(\mathbb{Z}_{|H|}) \right) ^{|P_{\ell}|}  \left( p_{\ell}^{-\frac{|P_{\ell}|(p_{\ell}-1)}{p_{\ell}}} \rho(\Z_{|P_{\ell}|}) \right)^{|H|} \\
        		&	=  \left( \prod_{i=1}^{\ell} p_i^{-\frac{(p_i-1)}{p_i}} \right)^{|G|} \rho(\mathbb{Z}_{|H|})^{|P_{\ell}|} \rho(\mathbb{Z}_{|P_{\ell}|})^{|H|} \\
        		& =  \left( \prod_{i=1}^ {\ell} p_i^{-\frac{(p_i-1)}{p_i}} \right)^{|G|}   \rho(\mathbb{Z}_{|G|})
        	\end{align*} 
        	This completes the proof. 	It is easy to see that when each $P_i$ is $ \mathbb{Z}_{p_i} \times \Z_{p_i}$, the equality holds. 
        \end{proof}

\section{More general groups} 	
\label{sec:moregeneral}

Figure~\ref{f:g60} shows the the Hasse dagram for the partially ordered set of
the $13$ isomorphism types of groups of order~$60$, ordered by domination of
the order sequences. The numbers are those in the \textsf{GAP}
\texttt{SmallGroups} library.

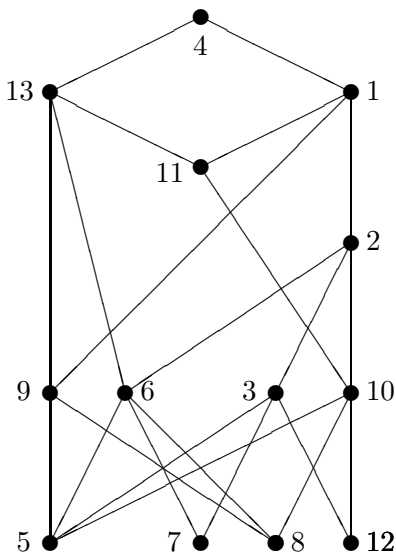
\begin{figure}[htbp]
\begin{center}
\setlength{\unitlength}{1mm}
\begin{picture}(40,70)
\put(20,70){\circle*{2}}\put(19,65){$4$}
\multiput(0,60)(40,0){2}{\circle*{2}}\put(-6,59){$13$}\put(42,59){$1$}
\put(20,50){\circle*{2}}\put(14,48){$11$}
\put(40,40){\circle*{2}}\put(42,39){$2$}
\multiput(0,20)(10,0){2}{\circle*{2}}\put(-4.5,19){$9$}\put(12,19){$6$}
\multiput(30,20)(10,0){2}{\circle*{2}}\put(25.5,19){$3$}\put(42,19){$10$}
\multiput(0,0)(20,0){3}{\circle*{2}}\put(-4.5,-1){$5$}
\put(15.5,-1){$7$}\put(42,-1){$12$}
\put(30,0){\circle*{2}}\put(32,-1){$8$}
\multiput(20,70)(20,-10){2}{\line(-2,-1){20}}
\multiput(20,70)(-20,-10){2}{\line(2,-1){20}}
\multiput(0,0)(40,0){2}{\line(0,1){60}}\put(42,-1){$12$}
\put(30,0){\circle*{2}}
\put(42,-1){$12$}
\put(0,60){\line(1,-4){10}}
\put(40,60){\line(-1,-1){40}}
\put(20,50){\line(2,-3){20}}
\put(40,40){\line(-3,-2){30}}
\put(40,40){\line(-1,-2){10}}
\put(0,20){\line(3,-2){30}}
\put(10,20){\line(-1,-2){10}}
\put(10,20){\line(1,-2){10}}
\put(10,20){\line(1,-1){20}}
\put(30,20){\line(-3,-2){30}}
\put(30,20){\line(-1,-2){10}}
\put(30,20){\line(1,-2){10}}
\put(40,20){\line(-2,-1){40}}
\put(40,20){\line(-1,-2){10}}
\end{picture}
\end{center}
\caption{\label{f:g60}Groups of order~$60$}
\end{figure}

Note that \texttt{SmallGroup(60,4)} is the cyclic group $\mathbb{Z}_{60}$,
\texttt{SmallGroup(60,13)} is $\mathbb{Z}_2\times\mathbb{Z}_{30}$ (the other
nilpotent group of order~$60$). Theorem~\ref{thm:non-nil-min-os-existence}
allows construction of three different non-nilpotent groups whose order
sequence is dominated by those of the two nilpotent groups; these are
\texttt{SmallGroup(60,$i$)} for $i=9,10,11$. Also, \texttt{SmallGroup(60,5)}
is the alternating group $A_5$, the unique non-solvable group of order~$60$;
it is minimal in the domination order but not unique with this property.

This example shows that the domination order will be difficult
to understand in general. Notice how abelian groups form only a very small
part of the poset, which as noted is a kind of non-abelian generalization
of the partition lattice.

\section{Connection with graphs}
\label{sec:graphs}

We present here some links between our problem and certain graphs associated
with finite groups.

The \emph{power graph} of a finite group $G$ has vertex set $G$, with an edge
$\{g,h\}$ if and only if one of $g$ and $h$ is a power of the other, that is,
either $h=g^m$ or $g=h^m$ for some integer $m$. For a recent survey of 
properties of this graph, see~\cite{kscc}.

\begin{theorem}
Suppose that $G$ and $H$ are finite groups with isomorphic power graphs. Then
$\os(G)=\os(H)$.
\end{theorem}

This is \cite[Corollary 3]{power2}. It can be seen as follows. The
\emph{directed power graph} of $G$ is the directed graph which has an arc from
$g$ to $h$ if and only if $h$ is a power of $G$. It is clear that the order of
an element $g$ is the out-degree of $g$ in this graph plus one (one for the
element itself). The main theorem of \cite{power2} asserts that the power graph
determines the directed power graph up to isomorphism.

As a consequence, \cite[Theorem 1]{power}
(stating that finite abelian groups are determined up to isomorphism by their
power graphs) is extended by Theorem~\ref{t:abelian} of this paper.

Pairs of groups with the same order sequence may or may not have isomorphic
power graphs. Examples of order~$16$ exhibit both behaviours: the $14$ groups
give rise to $12$ different power graphs and $9$ order sequences.
See~\cite{ms} for some results on groups with the same power graph.

\medskip

The \emph{Gruenberg--Kegel graph} of a finite group has vertices the prime
divisors of $G$, wth an edge $\{p,q\}$ if and only if $G$ contains an element
of order $pq$. This small graph contains a surprising amount of information
about the group $G$. Sometimes it is considered as a labelled graph, with
each vertex labelled by the corresponding prime. See~\cite{cm} for a recent
survey, and \cite{higman,Brandl-shi} for interesting earlier results.

\begin{proposition}
Let $G$ and $H$ be finite groups with $\os(G)=\os(H)$. Then the labelled
Gruenberg--Kegel graphs of $G$ and $H$ are equal.
\end{proposition}

This result is obvious from the definitions.

\section{Open questions}
\label{sec:more-questions}

\begin{question}
Investigate the relation of strong domination on groups of given order $n$.
\end{question} 

We note that domination and strong domination coincide for the case when $n$
is a prime power, since in that case the relations of order and divisibility
coincide on the divisors of $n$.

\begin{question}
Is the condition that $G$ is abelian necessary in Theorem~\ref{t:extension}?
\end{question}

We have already seen that if a group $G$ has the same order sequence as a cyclic group, then $G$ must be cyclic; whereas a non-abelian group can have the same order sequence as an abelian group. 
Moreover, nilpotency is characterised by the order sequence, and if there is a non-nilpotent group of order $n$ then a nilpotent group cannot be minimal.

What happens for solvable groups? 

\begin{question}
Is it true that a group having the same order sequence as a solvable group
is solvable?
\end{question}

We could ask the same question with ``supersolvable'' in place of ``solvable''.

\begin{question}
If there is a non-solvable group of order $n$, is it true that at least one group of order $n$ with minimal order sequence is non-solvable?
\end{question}

Note that there is a non-solvable group of order $n$ if and only if $n$ is a multiple of the order of a minimal (non-abelian) simple group; these groups were determined by Thompson~\cite{ngroups}. The case $n=60$ shown in Figure~\ref{f:g60} shows that, unlike for nilpotency, there will not be a non-solvable group which is dominated by every solvable group.

\begin{question}
Let $G$ and $H$ be two non-isomorphic (non-abelian) simple groups of the same order. Is it true that either $\os(G)$ dominates $\os(H)$ or \emph{vice versa}?
\end{question}

From the Classification of Finite Simple Groups, it is known that the only pairs of simple groups of the same order are
\begin{itemize}
\item $A_8$ and $\mathrm{PSL}(3,4)$; and
\item $\mathrm{PSp}(2n,q)$ and $\mathrm{P}\Omega(2n+1,q)$, where $q$ is an odd
prime power and $n\ge3$.
\end{itemize}
For the first pair, the $\mathbb{ATLAS}$ of Finite Groups~\cite{atlas} shows
that $\os(A_8)$ dominates $\os(\mathrm{PSL}(3,4))$. The answer is not known
in the other cases.

\begin{question}
Given a sequence of $n$ natural numbers, is it the order sequence of a group? If it is, then how many groups, and can we construct them? (For computational complexity reasons, it is better to take the input to be the collected order sequence, the set of pairs $(m,s(m))$ where $m$ is the order of an element and $s(m)$ the number of elements of this order, since this only requires a polylogarithmic amount of data.
\end{question} 

\begin{question}
Investigate further the poset of order sequences of finite groups.
\end{question}

\begin{question}
Given a finite poset $P$, find the minimum value of $n$ such that $P$ is
embeddable in the order sequence poset of groups of order~$n$.

We can ask the same question restricting to abelian groups.
\end{question}

For example, the $2$-element antichain is represented by two groups of
order $12$, or two abelian groups of order $36$; these are the smallest
possible orders.

\subsection*{Acknowledgements}
The second author would like to thank Prof.~Angsuman Das, Archita Mondal, Prof.~Sivaramakrishnan Sivasubramanian, and Prof.~Manoj Kumar Yadav
for helpful discussions. 
The second author acknowledges a SERB-National Post Doctoral Fellowship (File PDF/2021/001899) during
the preparation of this work and profusely thanks Science and Engineering Research Board, Government of India for this
funding. The second author also acknowledges excellent working conditions in the Department of Mathematics, Indian Institute
of Science.


\begin{thebibliography}{20}
                	
\bibitem{amiri-algapplctn}
H. Amiri and S. M. Jafarian Amiri,
Sum of element orders on finite groups of the same order,
\textit{J. Algebra Appl.} \textbf{10} (2011), 187--190.

\bibitem{amiri-communication}
H. Amiri, S. M. Jafarian Amiri and I. M. Isaacs,
Sums of element orders in finite groups,
\textit{Commun. Algebra} \textbf{37} (2009), 2978--2980.

\bibitem{amiri-strongdom}
M. Amiri,
On a bijection between a ﬁnite group and cyclic group,
\textit{J. Pure Appl. Algebra} \textbf{228} (2024),  paper no. 107632

\bibitem{amiri-communication-secondmax}
S. M. Jafarian Amiri,
Second maximum sum of element orders on finite nilpotent groups,
\textit{Commun. Algebra} \textbf{41} (2013), 2055--2059.
                	
\bibitem{amiri-pureandApplied} S. M. Jafarian Amiri and M. Amiri,
Second maximum sum of element orders on finite groups,
\textit{J. Pure Appl. Algebra} \textbf{218} (2014), 531--539.
                	
\bibitem{amiri-amiri-cia} S. M. Jafarian Amiri and M. Amiri,
Sum of the products of the orders of two distinct elements in finite groups,
\textit{Commun. Algebra} \textbf{42} (2014), 5319--5328.
                	
\bibitem{ber}
A. Ballester-Bolinches, R. Esteban-Romero and Derek J. S. Robinson,
On finite minimal non-nilpotent groups,
\textit{Proc. Amer. Math. Soc.}, \textbf{133} (2005), 3455--3462.

\bibitem{asad-joa}
M. Baniasad Asad and B. Khosravi,
A criterion for solvability of a finite group by the sum of element orders,
\textit{J. Algebra} \textbf{516} (2018), 115--124.

 \bibitem{Brandl-shi}
 R. Brandl and Shi Wujie,
 Finite groups whose element orders are consecutive integers,
 \textit{J. Algebra}, \textbf{143} (1991), 388--400.
                	
\bibitem{brylawski}
T. Brylawski,
The lattice of integer partitions,
\textit{Discrete Math.} \textbf{6} (1973), 201--219.

\bibitem{power2}
Peter J. Cameron,
The power graph of a finite group, II,
\textit{J. Group Theory} \textbf{13} (2010), 779--783.

\bibitem{power}
Peter J. Cameron and Shamik Ghosh,
The power graph of a finite group,
\textit{Discrete Math.} \textbf{311} (2011), 1220--1222.

\bibitem{cm}
Peter J. Cameron and Natalia Maslova,
Criterion of unrecognizability of a finite group by its Gruenberg--Kegel graph,
\textit{J. Algebra} \textbf{607} (2022), 186--213.

\bibitem{chew-chin-lim}
C. Y. Chew, A. Y. M. Chin, and C. S. Lim,
A recursive formula for the sum of element orders of finite abelian groups,
\textit{Results in Mathematics} \textbf{72} (2017), 1897--1905. 

\bibitem{atlas}
J.~H.~Conway, R.~T.~Curtis, S.~P.~Norton, R.~A.~Parker and R.~A.~Wilson,
\textit{$\mathbb{ATLAS}$ of Finite Groups}, Clarendon Press, Oxford, 1985.

\bibitem{Domenico-Monetta-Noce}
E. D. Domenico, C. Monetta and M. Noce,           
Upper bounds for the product of element orders of finite groups,
\textit{J. Algebraic Combinatorics} \textbf{57} (2023), 1033--1043.

\bibitem{gap}
The GAP Group,
GAP – Groups, Algorithms, and Programming, Version 4.11.0; 2020,
\url{https://www.gap-system.org}

\bibitem{garonzi-patassini}
M. Garonzi and M. Patassini,
Inequalities detecting structural properties of a finite group,
\textit{Commun. Algebra} \textbf{45} (2016) 677–-687.

\bibitem{greene}
C. Greene and D. J. Kleitman,
Longest Chains in the Lattice of Integer Partitions ordered by
Majorization,
\textit{Europ. J. Combinatorics}
\textbf{7} (1986), 1--10. 

\bibitem{hall}
Marshall Hall, Jr.,
\textit{The Theory of Groups},
Macmillan, New York, 1959.

\bibitem{Her-pureandApplied}
M. Herzog, P. Longobardi and M. Maj,
An exact upper bound for sums of element orders in non-cyclic finite groups,
\textit{J. Pure Appl. Algebra} \textbf{222} (2018), 1628--1642.
                	
\bibitem{Her-joa}
M. Herzog, P. Longobardi and M. Maj,
Two new criteria for solvability of finite groups,
\textit{J. Algebra} \textbf{511} (2018), 215--226.
                	
\bibitem{Her-cia}
M. Herzog, P. Longobardi and M. Maj,
Sums of element orders in groups of order $2m$ with $m$ odd,
\textit{Commun. Algebra} \textbf{47} (2019), 2035--2948.
                	
\bibitem{Her-jpaa}
M. Herzog, P. Longobardi and M. Maj,
The second maximal groups with respect to the sum of element orders,
\textit{Journal of Pure and Applied Algebra} \textbf{225} (2021), 106531.

\bibitem{higman}
G. Higman,
Finite groups in which every element has prime power order, 
\textit{J. London Math. Soc.} \textbf{32} (1957), 335--342.
                	
\bibitem{Isac-ams}
I. M. Isaacs,
\emph{Finite Group Theory,}
Graduate Studies in mathematics, Vol. 92, American Mathematical Society, Providence, RI, 2008.

\bibitem{kscc}
Ajay Kumar, Lavanya Selvaganesh, Peter J. Cameron and T. Tamizh Chelvam,
Recent developments on the power graph of finite groups -- a survey, 
\textit{AKCE Internat. J. Graphs Combinatorics} \textbf{18} (2021), 65--94.

\bibitem{ladisch}
F. Ladisch, Order-increasing bijection from arbitrary groups to cyclic groups,
\url{https://mathoverflow.net/a/107395}.

\bibitem{macdonald}
I. G. Macdonald,
\textit{Symmetric Functions and Hall Polynomials}, 
Clarendon Press, Oxford, 1977.

\bibitem{mazurov-khukhro}
V. D. Mazurov, E. I. Khukhro, The Kourovka Notebook. Unsolved Problems in Group Theory,
18th ed., Institute of Mathematics, Russian Academy of Sciences, Siberrian Division, Novosibirsk,
arXiv:1401.0300v3 [math. GR] (2014).

\bibitem{ronald}
R. McHaffey, Isomorphism of Finite Abelian Groups, \textit{The American Mathematical Monthly}, \textbf{72} (1965), 48--50. 

                	
\bibitem{ms}
M. Mirzargar and R. Scapellato,
Finite groups with the same power graph,
\textit{Commun. Algebra} \textbf{50} (2022), 1400--1406.


\bibitem{schmidt}
O. Yu. Schmidt,
Groups all of whose subgroups are nilpotent,
\textit{Mat. Sbornik} (Russian), \textbf{31} (1924), 366--372.                

\bibitem{Scott}
W. R. Scott.
\emph{Group Theory,}
Prentice-Hall, Englewood Cliffs, NJ, 1964.
                	
\bibitem{shen-et-al}
R. Shen, G. Chen, and C. Wu.
On groups with the second largest value of the sum of element orders,
\textit{Commun. Algebra} \textbf{43} (2015), 2618-2631.
                	
\bibitem{tarnauceanu-israel}
M. T\~{a}rn\~{a}uceanu, 
Detecting structural properties of finite groups by the sum of element orders,
\textit{Israel J. Math.} \textbf{238} (2020), 629–-637. 

\bibitem{ngroups}
J. G. Thompson, 
Nonsolvable finite groups all of whose local subgroups are solvable (Part I),
\textit{Bull. Amer. Math. Soc. (NS)}, \textbf{74} (1968), 383--437.

\end{thebibliography}
\end{document}